\documentclass[preprint,12pt]{elsarticle}

\usepackage[T1]{fontenc}
\usepackage{lmodern}
\usepackage{amsmath,amssymb,amsthm,mathtools}
\usepackage{enumitem}
\usepackage{microtype}
\usepackage[hidelinks]{hyperref}
\hypersetup{pdftitle={Singular Submodules of Abelian Groups over Their Endomorphism Rings},pdfauthor={Jiang Yang and Xin Zhang}}

\journal{Journal of Algebra}

\newtheorem{theorem}{Theorem}[section]
\newtheorem{proposition}[theorem]{Proposition}
\newtheorem{lemma}[theorem]{Lemma}
\newtheorem{corollary}[theorem]{Corollary}
\theoremstyle{remark}
\newtheorem{remark}[theorem]{Remark}
\theoremstyle{definition}
\newtheorem{definition}[theorem]{Definition}
\newtheorem{example}[theorem]{Example}

\newcommand{\Endo}{\operatorname{End}}
\newcommand{\Homm}{\operatorname{Hom}}
\newcommand{\Extt}{\operatorname{Ext}}
\newcommand{\Jaco}{\operatorname{Jac}}
\newcommand{\Anno}{\operatorname{Ann}}
\newcommand{\imma}{\operatorname{im}}
\newcommand{\SingSubo}{\operatorname{SingSub}}
\newcommand{\FIo}{\operatorname{FI}}

\newcommand{\Z}{\mathbb Z}
\newcommand{\Q}{\mathbb Q}
\newcommand{\Fp}{\mathbb F_p}
\newcommand{\leess}{\leq_{\mathrm e}}
\newcommand{\Sub}{\operatorname{Sub}}

\begin{document}

\begin{frontmatter}

\title{Singular Submodules of Abelian Groups over Their Endomorphism Rings}

\author[aff1]{Jiang Yang\corref{cor1}}
\ead{yangjiangdy@126.com}
\author[aff1]{Xin Zhang}
\ead{15247161004@163.com}
\cortext[cor1]{Corresponding author}
\address[aff1]{School of Mathematical Sciences, Guangxi Minzu University, Nanning, China}

\begin{abstract}
Let $A$ be an abelian group and $E=\Endo_{\Z}(A)$.  We give a module-theoretic description of the singular $E$-submodules asked for in Fuchs' Problem~1.2.  For a unital ring $R$ and a left $R$-module $M$, put $W=I_R({}_RR)\oplus I_R(M)$ and $J=\Jaco(\Endo_R(W))$, and let $u$ be the image of $1_R$.  The classical essential-kernel criterion yields
\[
        Z_R(M)=M\cap Ju.
\]
For $R=E$ and $M=A$, all singular submodules are therefore the $E$-submodules of $A\cap Ju$.  Writing $T=t(A)$ and $B=A/T$, we prove a torsion-transfer formula, identify the torsion part as $\bigoplus_p pT_p$, and describe simultaneous prime lifting by a canonical obstruction.  The resulting extension gives an $\Extt/\Homm$ parametrization of all singular submodules.  We obtain explicit formulas for torsion groups and for $\Z(p^\infty)\oplus B$ with $B$ torsion-free; in the latter case the fully invariant subgroup lattice of $B$ occurs as an interval.  The general description retains the induced endomorphism action and extension data, rather than giving a classification by classical group invariants.
\end{abstract}

\begin{keyword}
abelian group \sep endomorphism ring \sep singular submodule \sep injective envelope \sep fully invariant subgroup
\MSC[2020] 20K30 \sep 16D80 \sep 16S50
\end{keyword}

\end{frontmatter}

\section{Introduction}

An abelian group $A$ is naturally a faithful left module over $E=\Endo_{\Z}(A)$, and its $E$-submodules are precisely its fully invariant subgroups.  Fuchs discusses this point of view in \cite[Chapter~16, \S6]{Fuchs2015}.  At the end of Chapter~1 he asks \cite[Problem~1.2, p.~41]{Fuchs2015}:
\begin{quote}
\emph{What are the singular submodules of $A$ over its endomorphism ring $\Endo A$?}
\end{quote}
Here singularity is taken over the fixed ring $E$: an element $a\in A$ is singular when $\Anno_E(a)$ is an essential left ideal of $E$.  The problem places no restriction on $A$ and specifies no particular form for the desired description.  It is adjacent to the problem of characterizing fully invariant subgroup lattices of torsion-free groups \cite[Problem~1.1]{Fuchs2015}.

The singular submodule and its relation to injective envelopes belong to classical module theory; see \cite{Lam1999}.  The radical of the endomorphism ring of a quasi-injective module is described by essential kernels in \cite[Lemma~7, p.~897]{Osofsky1968}.  We use its injective case and include a proof.  The representation below is a consequence of this classical criterion.

Singularity over the endomorphism ring has also been studied explicitly.  Lee defines endosingular elements and endo-nonsingular modules in \cite[Definitions~5.3.1--5.3.2, p.~139]{Lee2010}, proves that the endosingular elements form a fully invariant submodule in \cite[Proposition~5.3.5]{Lee2010}, and gives cyclic and Pr\"ufer group examples in \cite[Examples~5.3.11 and~5.3.19]{Lee2010}.  We use this terminology for the established notion.

Another relevant line of work concerns groups that are torsion over their own endomorphism rings.  Faticoni constructs such torsion-free groups in \cite{Faticoni1994}.  Hill, Hill and Ullery \cite{HillHillUllery2001} use the Lambek torsion theory.  In particular, their Theorem~3 characterizes the torsion abelian groups that are Lambek torsion over their endomorphism rings as the divisible groups.  Lambek torsion implies singularity, but the converse fails for general modules.  We prove this distinction in Proposition~\ref{prop:lambek-comparison} and compare the torsion-group conclusions after Corollary~\ref{cor:torsion-extremes}.

We first give a representation valid for every unital ring $R$ and every left $R$-module $M$.  Let
\[
U=I_R({}_RR),\qquad V=I_R(M),\qquad W=U\oplus V,
\]
let $S=\Endo_R(W)$ and $J=\Jaco(S)$, and write $u\in U$ for the image of $1_R$ under the essential embedding ${}_RR\leess U$.  We prove
\begin{equation}\label{eq:intro-universal}
        Z_R(M)=M\cap Ju.
\end{equation}
Taking $R=E$ and $M=A$ identifies the largest singular $E$-submodule as $A\cap Ju$.  Thus every singular $E$-submodule is obtained by taking an $E$-submodule of this explicitly specified ambient submodule.  This answers the stated question in the sense of a uniform module-theoretic characterization for all abelian groups.  The description uses injective envelopes and a Jacobson radical; it does not compute the answer from rank, type or other classical group invariants.

We next separate the torsion and quotient data.  For $T=t(A)$ and $B=A/T$, we prove
\begin{equation}\label{eq:intro-transfer}
Z_E(A)=P_A\cap \pi^{-1}(Z_E(B)),
\qquad
P_A=\bigcap_{p:T_p\ne0}pA,
\end{equation}
where $B$ carries the action induced by $E$; in particular, $Z_E(B)$ must not in general be replaced by $Z_{\Endo(B)}(B)$.  We describe the induced singularity on $B$ in terms of liftable left ideals, and we construct an $E$-homomorphism
\[
\delta_A:\bigcap_{p:T_p\ne0}pB\longrightarrow
\frac{\prod_{p:T_p\ne0}T_p/pT_p}{\bigoplus_{p:T_p\ne0}T_p/pT_p}
\]
whose kernel is precisely $\pi(P_A)$.  This yields a canonical short exact sequence for $Z_E(A)$ and, through its extension class, an $\Extt/\Homm$ parametrization of all singular submodules.

The quotient term $Z_E(B)$ is retained in this description.  If $A$ is torsion-free, the torsion-transfer and extension formulas reduce to identities; the representation \eqref{eq:intro-universal} still applies, but these formulas do not give a further intrinsic classification in that case.  For torsion groups we obtain the explicit formula $Z_E(A)=\bigoplus_p pA_p$.

Finally, for every torsion-free $B$ and every prime $p$, we compute the singular submodule of $\Z(p^\infty)\oplus B$ and realize $\FIo(B)$ as an interval in its singular-submodule lattice.  This is a precise connection with the neighbouring lattice problem.  The contribution of the paper is the representation together with the transfer, lifting and extension formulas and these explicit applications.

\section{Preliminaries}

All rings are associative with identity, and all modules are unital left modules unless explicitly stated otherwise.  Endomorphisms are composed from right to left.  For a left $R$-module $M$ and $m\in M$, write
\[
\Anno_R(m)=\{r\in R:rm=0\}.
\]
A submodule $N\le M$ is \emph{essential}, written $N\leess M$, if $N\cap X\ne0$ for every nonzero submodule $X\le M$.  In particular, a left ideal $I\le R$ is essential when it is essential in the left regular module ${}_RR$.

\begin{definition}
The singular submodule of a left $R$-module $M$ is
\[
Z_R(M)=\{m\in M:\Anno_R(m)\leess R\}.
\]
The module $M$ is singular if $Z_R(M)=M$, and nonsingular if $Z_R(M)=0$.
\end{definition}

We use $I_R(M)$ for an injective envelope of $M$.  Background on essential extensions, injective envelopes, Jacobson radicals and singular modules may be found in \cite{Lam1999}.  We record two elementary facts that will be used repeatedly.

\begin{lemma}\label{lem:sing-basic}
For every left $R$-module $M$, the set $Z_R(M)$ is a fully invariant $R$-submodule of $M$.
\end{lemma}

\begin{proof}
If $x,y\in Z_R(M)$, then $\Anno_R(x)\cap\Anno_R(y)$ is an essential left ideal contained in $\Anno_R(x+y)$, so $x+y$ is singular.  Let $r\in R$ and let $0\ne L\le R$ be a left ideal.  If $Lr=0$, then $L\subseteq\Anno_R(rx)$; otherwise the nonzero left ideal $Lr$ meets $\Anno_R(x)$, so there is $l\in L$ with $0\ne lr\in\Anno_R(x)$, and hence $l(rx)=0$.  Thus $rx$ is singular.  Finally, if $\alpha:M\to M$ is $R$-linear, then $\Anno_R(x)\subseteq\Anno_R(\alpha x)$, and therefore $\alpha(Z_R(M))\subseteq Z_R(M)$.
\end{proof}

\begin{lemma}\label{lem:submodule-sing}
If $N\le_R M$, then
\[
Z_R(N)=N\cap Z_R(M).
\]
Consequently, $N$ is singular if and only if $N\le Z_R(M)$.
\end{lemma}

\begin{proof}
For $x\in N$, its annihilator is computed in the ring $R$ and is independent of whether $x$ is regarded as an element of $N$ or of $M$.  The displayed equality and the final assertion follow immediately.
\end{proof}

We use the Jacobson-radical criterion that $a\in\Jaco(S)$ if and only if $1-sa$ is invertible for every $s\in S$; see \cite{Lam1999}.  The next lemma is the injective case of the classical quasi-injective criterion in \cite[Lemma~7]{Osofsky1968}.  The right-module formulation there gives the left-module statement by replacing the base ring by its opposite.  We include the proof needed here.

\begin{lemma}[Classical radical criterion]\label{lem:rad-injective}
Let $W$ be an injective left $R$-module and $S=\Endo_R(W)$.  Then
\[
\Jaco(S)=\{f\in S:\ker f\leess W\}.
\]
\end{lemma}

\begin{proof}
Suppose first that $\ker f\leess W$, and let $g\in S$.  Put $h=1-gf$.  If $0\ne x\in\ker h$, then $x=gf(x)$.  Since $\ker f$ is essential, there is $r\in R$ with $0\ne rx\in\ker f$.  But
\[
rx=rgf(x)=g(rf(x))=0,
\]
a contradiction.  Thus $h$ is injective.  Because $W$ is injective, the monomorphism $h:W\to W$ splits.  Moreover $h$ restricts to the identity on $\ker f$, so $h(W)$ contains the essential submodule $\ker f$.  Hence the direct summand $h(W)$ is essential in $W$, and therefore $h(W)=W$.  Thus $1-gf$ is an automorphism for every $g\in S$, whence $f\in\Jaco(S)$.

Conversely, suppose that $\ker f$ is not essential.  Choose $0\ne X\le W$ with $X\cap\ker f=0$.  Then $f|_X:X\to f(X)$ is an isomorphism.  Since $W$ is injective, the inverse $f(X)\to X\hookrightarrow W$ extends to some $g\in S$.  Hence $(1-gf)X=0$, so $1-gf$ is not invertible.  Therefore $f\notin\Jaco(S)$.
\end{proof}

For an abelian group $A$, we write $E=\Endo_{\Z}(A)$ and regard $A$ as a left $E$-module by evaluation.  A subgroup $H\le A$ is an $E$-submodule if and only if it is fully invariant.  We denote by $\FIo(A)$ the lattice of fully invariant subgroups of $A$.

\subsection*{Comparison with Lambek torsion}

A left $R$-module $M$ is Lambek torsion if $\Homm_R(M,I_R({}_RR))=0$.  We shall use the following equivalent criterion from \cite[Proposition~1, p.~256]{HillHillUllery2001}: for every $x\in M$ and $0\ne c\in R$, there is $r\in R$ such that
\begin{equation}\label{eq:lambek-test}
rx=0,\qquad rc\ne0.
\end{equation}
This should be distinguished from the criterion for a fixed $x$ to be singular, which requires $(rc)x=0$ and $rc\ne0$ for each $0\ne c\in R$.

\begin{proposition}\label{prop:lambek-comparison}
Every Lambek torsion module is singular.  The converse need not hold.  If $R$ is a commutative integral domain, then the Lambek torsion modules, the singular modules and the ordinary torsion modules coincide.
\end{proposition}

\begin{proof}
Suppose $M$ is Lambek torsion.  Given $x\in M$ and $0\ne c\in R$, apply \eqref{eq:lambek-test} to $cx$.  There is $r$ such that $rc\ne0$ and $(rc)x=0$.  Hence $\Anno_R(x)$ meets every nonzero principal left ideal $Rc$, and is essential.

For failure of the converse, take $R=\Z/p^2\Z$ and $M=pR$.  Every nonzero element of $M$ has annihilator $pR$, an essential ideal of $R$, so $M$ is singular.  However, the inclusion $M\hookrightarrow R\hookrightarrow I_R(R)$ is nonzero, so $M$ is not Lambek torsion.

Finally, in a commutative integral domain every nonzero ideal is essential: for nonzero elements $a$ and $b$ in two nonzero ideals, $0\ne ab$ belongs to their intersection.  Thus singularity of $x$ is equivalent to a nonzero annihilator.  If every $x\in M$ has a nonzero annihilator, then a choice $0\ne r$ with $rx=0$ also satisfies $rc\ne0$ for every $0\ne c$, proving \eqref{eq:lambek-test}.
\end{proof}

\section{A universal representation of the singular submodule}

We derive the representation from Lemma~\ref{lem:rad-injective}.  The construction includes the injective envelope of the regular module as a direct summand; this is needed for the evaluation at the image of $1_R$.

\begin{lemma}\label{lem:JW-singular}
Let $R$ be a unital ring, $M$ a left $R$-module, and put
\[
U=I_R({}_RR),\quad V=I_R(M),\quad W=U\oplus V,
\quad S=\Endo_R(W),\quad J=\Jaco(S).
\]
Then $JW\subseteq Z_R(W)$.
\end{lemma}

\begin{proof}
The module $W$ is injective, being a finite direct sum of injective modules.
By Lemma~\ref{lem:rad-injective}, every $f\in J$ has essential kernel.  Let $w\in W$ and set $y=f(w)$.  To prove that $y$ is singular, let $0\ne L\le R$ be a left ideal.  If $Lw=0$, then $L\subseteq\Anno_R(y)$.  Otherwise $Lw$ is a nonzero submodule of $W$, hence
\[
0\ne Lw\cap\ker f.
\]
Choose $r\in L$ with $0\ne rw\in\ker f$.  Then
\[
ry=rf(w)=f(rw)=0,
\]
so $0\ne r\in L\cap\Anno_R(y)$.  Thus $\Anno_R(y)$ meets every nonzero left ideal of $R$, and hence is essential.  Therefore $f(w)\in Z_R(W)$.  Summing over $f\in J$ and $w\in W$ gives $JW\subseteq Z_R(W)$.
\end{proof}

Here $JW$ denotes the sum of the images $f(W)$ for $f\in J$.  Fix the essential embedding ${}_RR\leess U$, and write $u\in U$ for the image of $1_R$.  Thus $r\mapsto ru$ identifies ${}_RR$ with $Ru$, and $Ju=\{f(u):f\in J\}$ is an additive subgroup of $W$.

\begin{lemma}\label{lem:singular-Ju}
With the notation above,
\[
Z_R(W)\subseteq Ju.
\]
\end{lemma}

\begin{proof}
Let $x\in Z_R(W)$.  Define an $R$-homomorphism
\[
\varphi_x:Ru\longrightarrow W,
\qquad ru\longmapsto rx.
\]
Its kernel is $\Anno_R(x)u$, which is essential in $Ru$.  Since $Ru\leess U$, transitivity of essentiality shows that $\Anno_R(x)u\leess U$.  Because $W$ is injective, $\varphi_x$ extends to an $R$-homomorphism $\widetilde\varphi_x:U\to W$.

Define $F\in\Endo_R(W)$ by
\[
F|_U=\widetilde\varphi_x,
\qquad
F|_V=0.
\]
The kernel of $F|_U$ contains the essential submodule $\Anno_R(x)u$ of $U$, hence $\ker(F|_U)\leess U$.  It follows that
\[
\ker F=\ker(F|_U)\oplus V\leess U\oplus V=W.
\]
By Lemma~\ref{lem:rad-injective}, $F\in J$.  Finally,
\[
F(u)=\varphi_x(u)=x,
\]
so $x\in Ju$.
\end{proof}

\begin{theorem}[Universal representation]\label{thm:universal}
Let $R$ be a unital ring and $M$ a left $R$-module.  Choose injective envelopes
\[
U=I_R({}_RR),\qquad V=I_R(M),
\]
put
\[
W=U\oplus V,
\qquad S=\Endo_R(W),
\qquad J=\Jaco(S),
\]
and let $u\in U$ be the image of $1_R$ in the essential embedding ${}_RR\leess U$.  Then
\begin{equation}\label{eq:universal}
        Z_R(M)=M\cap Ju,
\end{equation}
where $M$ is identified with its image in $V\le W$.
\end{theorem}

\begin{proof}
Lemmas~\ref{lem:JW-singular} and \ref{lem:singular-Ju} give
\[
Z_R(W)=JW=Ju.
\]
By Lemma~\ref{lem:submodule-sing}, applied to the embedded submodule $M\le W$,
\[
Z_R(M)=M\cap Z_R(W)=M\cap Ju.
\]
\end{proof}

\begin{corollary}[Module-theoretic description for Problem~1.2]\label{cor:fuchs}
Let $A$ be any abelian group and $E=\Endo_{\Z}(A)$.  Let
\[
U=I_E({}_EE),\qquad V=I_E(A),\qquad W=U\oplus V,
\]
let $S=\Endo_E(W)$, $J=\Jaco(S)$, and let $u\in U$ be the image of $1_E$.  Then
\begin{equation}\label{eq:Fuchs-solution}
Z_E(A)=A\cap Ju.
\end{equation}
Moreover,
\begin{equation}\label{eq:sing-lattice}
\SingSubo_E(A)=\Sub_E(A\cap Ju).
\end{equation}
Equivalently, the singular submodules are precisely the fully invariant subgroups of $A$ that are contained in $A\cap Ju$.
\end{corollary}

\begin{proof}
Equation~\eqref{eq:Fuchs-solution} is Theorem~\ref{thm:universal} with $R=E$ and $M=A$.  Equation~\eqref{eq:sing-lattice} follows from Lemma~\ref{lem:submodule-sing}, and $E$-submodules of $A$ are exactly fully invariant subgroups.
\end{proof}

\begin{remark}
Corollary~\ref{cor:fuchs} covers every abelian group and every singular $E$-submodule, with no finiteness or splitting hypothesis.  It gives an exact characterization in terms of injective envelopes and a Jacobson radical.  These objects are formed from $E$ and $A$ without first knowing $Z_E(A)$.  Different choices of injective envelopes yield the same submodule after intersection with the embedded copy of $A$, by \eqref{eq:Fuchs-solution}.  This scope should be distinguished from a computation of the singular-submodule lattice in terms of classical invariants of $A$.
\end{remark}

\section{Torsion transfer and relative singularity on the quotient}

Let $A$ now be an arbitrary abelian group.  Throughout this section set
\[
E=\Endo(A),\qquad T=t(A),\qquad B=A/T,
\]
and let $\pi:A\to B$ be the canonical map.  Since $T$ is fully invariant, every element of $E$ induces an endomorphism of $B$, so $B$ is a left $E$-module.  This induced action is essential to the statements below.

Let
\[
\mathcal P_A=\{p\text{ prime}:T_p\ne0\},
\qquad
P_A=\bigcap_{p\in\mathcal P_A}pA,
\]
where the empty intersection is $A$, and put
\[
T^\circ=\bigoplus_p pT_p.
\]

We shall repeatedly use the following elementary reformulation of essentiality.

\begin{lemma}\label{lem:principal-test}
For $x\in A$, the following are equivalent:
\begin{enumerate}[label=\textup{(\roman*)}]
\item $x\in Z_E(A)$;
\item for every $0\ne f\in E$ there exists $g\in E$ such that $0\ne gf$ and $(gf)(x)=0$.
\end{enumerate}
\end{lemma}

\begin{proof}
The left ideal $\Anno_E(x)$ is essential if and only if it meets every nonzero principal left ideal $Ef$ nontrivially.
\end{proof}

\begin{theorem}[Torsion transfer]\label{thm:transfer}
With the notation above,
\begin{equation}\label{eq:transfer}
Z_E(A)=P_A\cap\pi^{-1}\bigl(Z_E(B)\bigr).
\end{equation}
Consequently,
\begin{equation}\label{eq:torsion-core}
Z_E(A)\cap T=T^\circ=\bigoplus_p pT_p.
\end{equation}
\end{theorem}

\begin{proof}
Let $x\in Z_E(A)$.  Since $\Anno_E(x)\subseteq\Anno_E(\pi x)$, the latter left ideal is essential, so $\pi x\in Z_E(B)$.

Fix $p\in\mathcal P_A$.  Suppose that $x\notin pA$.  Then $x+pA$ is nonzero in the $\mathbb F_p$-vector space $A/pA$.  Choose an $\mathbb F_p$-linear functional $\lambda:A/pA\to\mathbb F_p$ with $\lambda(x+pA)=1$.  Since $T_p\ne0$, choose $0\ne c\in T_p[p]$.  Identifying $\mathbb F_p$ with $\langle c\rangle$, the map
\[
h:A\longrightarrow \langle c\rangle\le A,
\qquad
h(a)=\lambda(a+pA)c,
\]
is an endomorphism with $h(x)=c$ and $h(A)=\langle h(x)\rangle$.  If $g\in E$ satisfies $(gh)(x)=0$, then $g$ annihilates $h(A)$, hence $gh=0$.  Thus
\[
Eh\cap\Anno_E(x)=0,
\]
contrary to the essentiality of $\Anno_E(x)$.  Therefore $x\in pA$.  Since $p$ was arbitrary, $x\in P_A$.  This proves the inclusion ``$\subseteq$'' in \eqref{eq:transfer}.

Conversely, let $x\in P_A$ and suppose that $\pi x\in Z_E(B)$.  Let $0\ne f\in E$.  Since $\Anno_E(\pi x)$ is essential, there exists
\[
0\ne h\in Ef\cap\Anno_E(\pi x).
\]
Then $h(x)\in T$.  If $h(x)=0$, we are done.  Otherwise let $n$ be the order of $h(x)$.  If $nh\ne0$, then $0\ne nh\in Ef$ and $(nh)(x)=0$.

Assume now that $nh=0$.  Then $h(A)$ is a nonzero torsion group of exponent dividing $n$.  Choose a prime $p$ for which its $p$-primary component is nonzero, and write $n=p^k m$ with $(p,m)=1$.  Choose an integer $c$ such that $c\equiv1\pmod{p^k}$ and $c\equiv0\pmod m$.  Multiplication by $c$ on $h(A)$ is the projection onto that component.  Thus $h_p=ch$ is a nonzero member of $Eh\subseteq Ef$, and $h_p(A)$ is a bounded $p$-group.  This uses only a scalar multiple of $h$ and does not require a projection of $A$ onto $T_p$.  Choose $e\ge1$ minimal with $p^e h_p=0$.  Then $p^{e-1}h_p\ne0$.  Since $h_p(A)\ne0$, we have $T_p\ne0$.  As $x\in P_A$, write $x=pz$.  Then
\[
(p^{e-1}h_p)(x)=p^e h_p(z)=0.
\]
Thus every nonzero principal left ideal $Ef$ meets $\Anno_E(x)$ nontrivially.  By Lemma~\ref{lem:principal-test}, $x\in Z_E(A)$, proving \eqref{eq:transfer}.

For \eqref{eq:torsion-core}, note first that for $t\in T$ and $p\in\mathcal P_A$,
\[
t\in pA\quad\Longleftrightarrow\quad t\in pT,
\]
because $t=pa$ implies $p\pi(a)=0$ in the torsion-free group $B$, hence $a\in T$.  Therefore
\[
T\cap P_A=\bigcap_{p\in\mathcal P_A}pT.
\]
On the $q$-primary component $T_q$, multiplication by $p$ is an automorphism for $p\ne q$, while the condition at $p=q$ gives $qT_q$.  Hence $T\cap P_A=\bigoplus_q qT_q=T^\circ$.  Since $0\in Z_E(B)$, equation \eqref{eq:transfer} now yields \eqref{eq:torsion-core}.
\end{proof}

The quotient term $Z_E(B)$ in Theorem~\ref{thm:transfer} is relative to the action induced by $E$.  We now make this dependence explicit.  Let
\[
\rho:E\longrightarrow\Endo(B)
\]
be the induced ring homomorphism, and put
\[
S_A=\imma\rho,
\qquad
H_A=\ker\rho=\{f\in E:f(A)\subseteq T\}.
\]
We identify $H_A$ additively with $\Homm(A,T)$.

\begin{definition}\label{def:liftable}
A left ideal $J\le S_A$ is \emph{liftable} if there exists a left ideal $L\le E$ such that
\[
L\cap H_A=0,
\qquad
\rho(L)=J.
\]
\end{definition}

\begin{proposition}[Relative essentiality]\label{prop:relative-essentiality}
Let $I\le S_A$ be a left ideal.  Then
\begin{equation}\label{eq:relative-essentiality}
\rho^{-1}(I)\leess E
\quad\Longleftrightarrow\quad
I\cap J\ne0
\text{ for every nonzero liftable left ideal }J\le S_A.
\end{equation}
Consequently,
\begin{equation}\label{eq:relative-sing}
Z_E(B)=\{b\in B:\Anno_{S_A}(b)\text{ meets every nonzero liftable left ideal of }S_A\}.
\end{equation}
\end{proposition}

\begin{proof}
Suppose first that $\rho^{-1}(I)\leess E$, and let $0\ne J\le S_A$ be liftable through a left ideal $L\le E$ with $L\cap H_A=0$.  Since $L\ne0$,
\[
0\ne L\cap\rho^{-1}(I).
\]
Choose $0\ne x$ in this intersection.  Then $\rho(x)\ne0$ because $L\cap H_A=0$, and $\rho(x)\in J\cap I$.

Conversely, suppose that $\rho^{-1}(I)$ is not essential.  Choose a nonzero left ideal $L\le E$ with
\[
L\cap\rho^{-1}(I)=0.
\]
Since $H_A\subseteq\rho^{-1}(I)$, we have $L\cap H_A=0$.  Thus $J=\rho(L)$ is a nonzero liftable left ideal.  If $y\in J\cap I$, write $y=\rho(x)$ with $x\in L$.  Then $x\in\rho^{-1}(I)$, so $x=0$ and $y=0$.  Hence $J\cap I=0$, proving the contrapositive of the reverse implication.

For \eqref{eq:relative-sing}, observe that
\[
\Anno_E(b)=\rho^{-1}(\Anno_{S_A}(b))
\]
for every $b\in B$, and apply \eqref{eq:relative-essentiality}.
\end{proof}

\begin{remark}\label{rem:no-swap}
Even though $B$ is torsion-free, the singular submodule in \eqref{eq:transfer} is $Z_E(B)$ for the action induced from $A$.  In general it cannot be replaced by $Z_{\Endo(B)}(B)$; see Example~\ref{ex:no-swap}.  Proposition~\ref{prop:relative-essentiality} records the relevant left-ideal information in the map $E\to S_A$.  It is an equivalence and does not, by itself, compute all liftable left ideals.
\end{remark}

\section{A simultaneous prime-lifting obstruction}

The term $P_A=\bigcap_{p\in\mathcal P_A}pA$ in Theorem~\ref{thm:transfer} contains a simultaneous divisibility condition across all primes occurring in the torsion subgroup.  We now express its image in $B=A/T$ by a canonical obstruction.

Set
\[
B_{\mathcal P}=\bigcap_{p\in\mathcal P_A}pB
\]
and
\[
D_T=
\frac{\displaystyle\prod_{p\in\mathcal P_A}T_p/pT_p}
     {\displaystyle\bigoplus_{p\in\mathcal P_A}T_p/pT_p}.
\]
Each endomorphism of $A$ preserves $T_p$ and $pT_p$, so $D_T$ is naturally an $E$-module.

For $b\in B_{\mathcal P}$ choose a lift $a\in A$ with $\pi(a)=b$.  For every $p\in\mathcal P_A$, choose $y_p\in A$ with
\[
p\pi(y_p)=b.
\]
Then $a-py_p\in T$.  Denote by $(a-py_p)_p$ its $p$-primary component and define
\begin{equation}\label{eq:delta-def}
\delta_A(b)=
\left[
\bigl((a-py_p)_p+pT_p\bigr)_{p\in\mathcal P_A}
\right]\in D_T.
\end{equation}

\begin{proposition}\label{prop:delta-well}
The map $\delta_A:B_{\mathcal P}\to D_T$ in \eqref{eq:delta-def} is a well-defined $E$-homomorphism.
\end{proposition}

\begin{proof}
Fix $p$.  If $y_p'$ is another choice, then
\[
p\pi(y_p'-y_p)=0.
\]
Since $B$ is torsion-free, $y_p'-y_p\in T$.  Hence the $p$-primary components of $a-py_p'$ and $a-py_p$ differ by an element of $pT_p$, so the $p$th residue class is independent of $y_p$.

If the lift $a$ is replaced by $a+t$ with $t\in T$, then the $p$th residue changes by $t_p+pT_p$.  Every element $t\in T=\bigoplus_pT_p$ has only finitely many nonzero primary components, so the total change lies in the direct sum
\[
\bigoplus_{p\in\mathcal P_A}T_p/pT_p.
\]
Thus the class in $D_T$ is independent of the lift $a$.

Additivity follows by choosing sums of lifts and sums of the corresponding $y_p$'s.  If $e\in E$, then $e(a)$ and $e(y_p)$ are valid choices for $e(b)$, and
\[
e(a)-pe(y_p)=e(a-py_p).
\]
Taking $p$-primary residues gives $\delta_A(eb)=e\delta_A(b)$.
\end{proof}

\begin{theorem}\label{thm:delta-kernel}
The obstruction $\delta_A$ satisfies
\begin{equation}\label{eq:delta-kernel}
\ker\delta_A=\pi(P_A).
\end{equation}
\end{theorem}

\begin{proof}
Fix $b\in B_{\mathcal P}$ and a lift $a\in A$.  For $p\in\mathcal P_A$, let
\[
c_p=(a-py_p)_p+pT_p.
\]
We claim that
\begin{equation}\label{eq:cp-test}
c_p=0\quad\Longleftrightarrow\quad a\in pA.
\end{equation}
Indeed, if $c_p=0$, then the $p$-primary component of $a-py_p$ belongs to $pT_p$.  On every $q$-primary component with $q\ne p$, multiplication by $p$ is an automorphism.  Hence $a-py_p\in pT$, so $a\in pA$.  Conversely, if $a=pz$, then
\[
a-py_p=p(z-y_p)\in T.
\]
Since $p\pi(z-y_p)=0$ and $B$ is torsion-free, $z-y_p\in T$, whence $a-py_p\in pT$ and $c_p=0$.

If $a\in P_A$, then \eqref{eq:cp-test} gives $c_p=0$ for every $p$, so $\delta_A(\pi a)=0$.  Hence $\pi(P_A)\subseteq\ker\delta_A$.

Conversely, suppose $\delta_A(b)=0$.  Then the family $(c_p)_p$ represents an element of the direct sum, so the set
\[
F=\{p\in\mathcal P_A:c_p\ne0\}
\]
is finite.  Choose representatives $u_p\in T_p$ of $c_p$ for $p\in F$ and put
\[
t=-\sum_{p\in F}u_p\in T,
\qquad a'=a+t.
\]
Replacing $a$ by $a'$ changes the $p$th residue by $t_p+pT_p$.  This cancels $c_p$ for $p\in F$ and leaves it zero for $p\notin F$.  Thus a single element $t\in T$ corrects all the residues.  By \eqref{eq:cp-test}, $a'\in pA$ for every $p\in\mathcal P_A$, hence $a'\in P_A$.  Since $\pi(a')=b$, we have $b\in\pi(P_A)$.
\end{proof}

\begin{remark}
If only finitely many primes occur in $T$, then $D_T=0$ and $\pi(P_A)=B_{\mathcal P}$.  With infinitely many primes, the obstruction can be nonzero, as the following example shows.
\end{remark}

\begin{example}\label{ex:product-primes}
Let $A=\prod_p\Z/p\Z$, where $p$ runs over all primes.  Then
\[
T=\bigoplus_p\Z/p\Z,\qquad T^\circ=0,\qquad P_A=0.
\]
Indeed, an element annihilated by a positive integer has support contained in the finite set of primes dividing that integer.  Also, $pA$ consists precisely of the elements whose $p$th coordinate is zero.

The quotient $B=A/T$ is torsion-free and divisible.  For divisibility, solve $ny=a$ coordinatewise at primes not dividing $n$, and put the remaining coordinates of $y$ equal to zero.  The difference $a-ny$ has finite support and hence lies in $T$.  Torsion-freeness follows because $na\in T$ forces $a$ to have finite support.  Thus $B_{\mathcal P}=B$ and $D_T$ is naturally $B$.

For each prime $p$, let $e_p\in E$ be the coordinate projection followed by inclusion.  If $0\ne f\in E$, some $e_pf$ is nonzero and has image in $T$.  Consequently $H_A\leess E$, and $Z_E(B)=B$.  On the other hand, $(py_p)_p=0$ in the definition of $\delta_A$, so $\delta_A(b)=b$ under the identification $D_T=B$.  Hence $\ker\delta_A=0$ and Theorem~\ref{thm:transfer} gives $Z_E(A)=0$.  Individual prime divisibility in $B$ therefore need not admit a simultaneous lift to $A$.
\end{example}

\section{An extension description of singular submodules}

Define
\begin{equation}\label{eq:QA}
Q_A=Z_E(B)\cap\ker\delta_A.
\end{equation}
Write $K_A=\ker\delta_A$.  By Theorem~\ref{thm:delta-kernel} and the computation $T\cap P_A=T^\circ$ in Theorem~\ref{thm:transfer}, there is a short exact sequence
\begin{equation}\label{eq:ambient-extension}
0\longrightarrow T^\circ\longrightarrow P_A
\overset{\pi}{\longrightarrow}K_A\longrightarrow0.
\end{equation}
Let $\eta_A\in\Extt_E^1(K_A,T^\circ)$ be its class, let $j_A:Q_A\hookrightarrow K_A$ be the inclusion, and define
\begin{equation}\label{eq:zeta-definition}
\zeta_A=j_A^*(\eta_A)\in\Extt_E^1(Q_A,T^\circ).
\end{equation}
Thus the extension class is defined by the lifting data before identifying its middle term with the singular submodule.

\begin{theorem}\label{thm:short-exact}
There is a canonical short exact sequence of left $E$-modules
\begin{equation}\label{eq:short-exact}
0\longrightarrow T^\circ
\longrightarrow Z_E(A)
\overset{\pi}{\longrightarrow}Q_A
\longrightarrow0.
\end{equation}
Its extension class is $\zeta_A$ from \eqref{eq:zeta-definition}.
\end{theorem}

\begin{proof}
The pullback of \eqref{eq:ambient-extension} along $j_A$ has middle term $P_A\cap\pi^{-1}(Q_A)$.  Since $\pi(P_A)=K_A$ and $Q_A=Z_E(B)\cap K_A$, Theorem~\ref{thm:transfer} identifies this middle term with $Z_E(A)$.  Exactness and the statement about the extension class follow.  In particular, surjectivity onto $Q_A$ follows from $\pi(P_A)=K_A$, and the kernel is $T^\circ$.
\end{proof}

For the next theorem we use the standard interpretation of $\Extt^1$ in terms of equivalence classes of short exact sequences; see \cite{Rotman2009}.  If $C\le_E T^\circ$, write
\[
q_C:T^\circ\longrightarrow T^\circ/C,
\]
and if $L\le_E Q_A$, write $i_L:L\hookrightarrow Q_A$.

\begin{theorem}[Ext/Hom parametrization]\label{thm:ExtHom}
Let $C\le_E T^\circ$ and $L\le_E Q_A$.  There exists an $E$-submodule $N\le Z_E(A)$ satisfying
\begin{equation}\label{eq:CL-data}
N\cap T^\circ=C,
\qquad
\pi(N)=L
\end{equation}
if and only if
\begin{equation}\label{eq:Ext-obstruction}
(q_C)_*(i_L)^*(\zeta_A)=0
\quad\text{in}\quad
\Extt_E^1(L,T^\circ/C).
\end{equation}
Whenever \eqref{eq:Ext-obstruction} holds, the set of all $N$ satisfying \eqref{eq:CL-data} is a torsor under
\begin{equation}\label{eq:Hom-torsor}
\Homm_E(L,T^\circ/C).
\end{equation}
Consequently, the singular $E$-submodules of $A$ are completely described by the pairs $(C,L)$ satisfying \eqref{eq:Ext-obstruction}, together with a choice of splitting of the corresponding extension.
\end{theorem}

\begin{proof}
Set $Z_L=P_A\cap\pi^{-1}(L)$.  As $L\le Q_A$, Theorem~\ref{thm:transfer} gives $Z_L\le Z_E(A)$.  Pulling back the sequence of Theorem~\ref{thm:short-exact} along $i_L$ gives
\[
0\to T^\circ\to Z_L\to L\to0,
\]
and then push out along $q_C$ to obtain
\begin{equation}\label{eq:pushed}
0\to T^\circ/C\to Z_L/C\to L\to0.
\end{equation}
Its extension class is exactly $(q_C)_*(i_L)^*(\zeta_A)$.

Suppose that $N$ satisfies \eqref{eq:CL-data}.  Then $N\le Z_L$, and $N/C$ maps isomorphically onto $L$ while meeting $T^\circ/C$ trivially.  Hence $N/C$ is a complement to $T^\circ/C$ in \eqref{eq:pushed}; in particular, \eqref{eq:pushed} splits and its class vanishes.

Conversely, if the class vanishes, choose an $E$-linear splitting $s:L\to Z_L/C$ of \eqref{eq:pushed}.  Let $N$ be the inverse image of $s(L)$ under $Z_L\to Z_L/C$.  Then $N$ is an $E$-submodule, $N\cap T^\circ=C$, and $\pi(N)=L$.

These constructions are inverse.  In fact, for a given $N$, the restriction of $Z_L/C\to L$ to $N/C$ is an isomorphism, so its inverse is the unique splitting with image $N/C$.  Thus two splittings give the same $N$ only if they agree.  The difference of two splittings is an $E$-homomorphism $L\to T^\circ/C$, and adding any such homomorphism to one splitting gives another.  The action of $\Homm_E(L,T^\circ/C)$ on the splittings, and hence on the corresponding submodules, is therefore free and transitive.  No preferred splitting is asserted.
\end{proof}

\begin{corollary}\label{cor:structural-answer}
For an arbitrary abelian group $A$, the lattice of singular $\Endo(A)$-submodules is determined by the following data:
\begin{enumerate}[label=\textup{(\roman*)}]
\item the torsion core $T^\circ=\bigoplus_p pT_p$;
\item the relative singular submodule $Z_E(B)$, equivalently the liftable-ideal condition \eqref{eq:relative-sing};
\item the simultaneous prime-lifting obstruction $\delta_A$;
\item the extension class $\zeta_A\in\Extt_E^1(Q_A,T^\circ)$.
\end{enumerate}
More precisely, all singular submodules are given by Theorem~\ref{thm:ExtHom}.
\end{corollary}

\begin{remark}\label{rem:torsion-free-boundary}
If $T=0$, then $P_A=A$, $B=A$, $H_A=0$, $S_A=E$, $\delta_A=0$, and $Q_A=Z_E(A)$.  Every left ideal is liftable.  Thus the relative-essentiality criterion becomes the original singularity criterion, and \eqref{eq:short-exact} becomes
\[
0\longrightarrow0\longrightarrow Z_E(A)
\overset{\mathrm{id}}{\longrightarrow}Z_E(A)\longrightarrow0.
\]
The extension description is exact for all $A$, but in this case gives no further determination of $Z_E(A)$.  Corollary~\ref{cor:fuchs} remains a general representation, whereas a more explicit treatment of arbitrary torsion-free groups requires additional information about their endomorphism action.
\end{remark}

\section{Consequences and test cases}

We record several immediate consequences.  Besides being useful in their own right, they provide consistency checks for the general formulas.

\subsection{Torsion groups}

\begin{corollary}\label{cor:torsion}
If $A$ is a torsion abelian group, then
\begin{equation}\label{eq:torsion-formula}
Z_{\Endo(A)}(A)=\bigoplus_p pA_p.
\end{equation}
Hence the singular $\Endo(A)$-submodules are precisely the fully invariant subgroups contained in $\bigoplus_p pA_p$.
\end{corollary}

\begin{proof}
Here $T=A$ and $B=0$.  Equation \eqref{eq:torsion-core} gives \eqref{eq:torsion-formula}; the statement about all singular submodules follows from Lemma~\ref{lem:submodule-sing}.
\end{proof}

Following \cite[Definition~5.3.2]{Lee2010}, call $A$ \emph{endo-nonsingular} if $Z_{\Endo(A)}(A)=0$ and \emph{endosingular} if $Z_{\Endo(A)}(A)=A$.  A torsion group is called elementary here if each $p$-component is an $\Fp$-vector space; different primes are allowed.

\begin{corollary}\label{cor:torsion-extremes}
Let $A$ be torsion.
\begin{enumerate}[label=\textup{(\roman*)}]
\item $A$ is endo-nonsingular if and only if $A$ is elementary.
\item $A$ is endosingular if and only if $A$ is divisible.
\end{enumerate}
\end{corollary}

\begin{proof}
By \eqref{eq:torsion-formula}, $Z_{\Endo(A)}(A)=0$ exactly when $pA_p=0$ for every $p$, which is equivalent to each $A_p$ being elementary.  Likewise $Z_{\Endo(A)}(A)=A$ exactly when $pA_p=A_p$ for every $p$, i.e. each $p$-component is divisible.
\end{proof}

Hill, Hill and Ullery prove that a torsion abelian group is Lambek torsion over its endomorphism ring if and only if it is divisible \cite[Theorem~3]{HillHillUllery2001}.  Together with Corollary~\ref{cor:torsion-extremes}, this shows that, for torsion abelian groups, being endosingular and being Lambek torsion over the endomorphism ring are equivalent conditions on the whole group.  This comparison does not identify the two notions for arbitrary modules; Proposition~\ref{prop:lambek-comparison} gives a counterexample.  Corollary~\ref{cor:torsion} determines the entire singular submodule, including when it is a proper nonzero subgroup.

For a $p$-group $A$, the formula is $Z_{\Endo(A)}(A)=pA$.  In particular, it gives $Z_{\Endo(\Z/4\Z)}(\Z/4\Z)=2(\Z/4\Z)$, the example in \cite[Example~5.3.11]{Lee2010}, and gives $Z_{\Endo(\Z(p^\infty))}(\Z(p^\infty))=\Z(p^\infty)$, consistent with \cite[Example~5.3.19]{Lee2010}.

\subsection{Rank-one torsion-free groups}

\begin{corollary}\label{cor:rank-one}
Every nonzero torsion-free abelian group of rank one is endo-nonsingular.
\end{corollary}

\begin{proof}
Such a group $A$ embeds in $\Q$, and every endomorphism is multiplication by a rational number belonging to a subring $E\le\Q$.  If $0\ne a\in A$, then $fa=0$ implies $f=0$.  Thus $\Anno_E(a)=0$.  Since $E\ne0$, the zero left ideal is not essential in $E$.  Hence no nonzero element of $A$ is singular.
\end{proof}

The rank-one restriction is substantial.  For example, taking $S=\Z$ in \cite[Example~3.11]{Faticoni1994} gives a torsion-free abelian group $A$ with $\Endo(A)=\Z[[X]]$ such that every element is annihilated by a power of $X$.  Since $\Z[[X]]$ is a commutative integral domain, Proposition~\ref{prop:lambek-comparison} shows that this group is endosingular.

\section{The fully invariant subgroup lattice inside Problem~1.2}

The next theorem realizes the fully invariant subgroup lattice of an arbitrary torsion-free group as an interval in a singular-submodule lattice.  It gives a concrete link between Problems~1.1 and~1.2 of \cite{Fuchs2015}.

Fix a prime $p$ and put
\[
D=\Z(p^\infty).
\]
Recall that $D$ is divisible, hence injective as an abelian group.

\begin{theorem}\label{thm:FI-embedding}
Let $B$ be any torsion-free abelian group and let
\[
A=D\oplus B.
\]
Then
\begin{equation}\label{eq:DB-sing}
Z_{\Endo(A)}(A)=D\oplus pB.
\end{equation}
Moreover, the interval of singular submodules containing $D$ is isomorphic, as a lattice, to the fully invariant subgroup lattice of $B$.  Explicitly,
\begin{equation}\label{eq:lattice-embed}
[D,\,Z_{\Endo(A)}(A)]\cong \FIo(B).
\end{equation}
\end{theorem}

\begin{proof}
Let $E=\Endo(A)$ and $T=D$.  Since $T$ has only the $p$-primary component and $pD=D$, we have
\[
P_A=pA=D\oplus pB.
\]
It remains, by Theorem~\ref{thm:transfer}, to show that every element of $B=A/D$ is singular for the induced $E$-action.

Let $H=\ker(E\to\Endo(B))$, the left ideal of endomorphisms with image in $D$.  We claim that $H\leess E$.  Let $0\ne f\in E$, and regard the projection $e_D$ as an endomorphism of $A$ by composing it with the inclusion of $D$.  If $e_Df\ne0$, then $0\ne e_Df\in Ef\cap H$.  Otherwise $f(A)\subseteq B$, and we can choose $x\in A$ with $0\ne f(x)\in B$.  Torsion-freeness gives $\langle f(x)\rangle\cong\Z$.  Send $f(x)$ to a fixed nonzero element of $D$ and extend this homomorphism to $\varphi:B\to D$ by injectivity.  Define $g\in E$ by $g(d,b)=(\varphi(b),0)$.  Then $(gf)(x)\ne0$ and $gf\in H$.  In either case $Ef\cap H\ne0$, proving the claim.

For every $b\in B$, $H\subseteq\Anno_E(b)$, so $b\in Z_E(B)$.  Hence Theorem~\ref{thm:transfer} gives \eqref{eq:DB-sing}.

Now $E\to\Endo(B)$ is surjective because every endomorphism of $B$ lifts diagonally to $D\oplus B$.  Therefore the $E$-submodules $N$ with
\[
D\le N\le D\oplus pB
\]
correspond, after quotienting by $D$, exactly to the $\Endo(B)$-submodules of $pB$.  Since multiplication by $p$ gives an $\Endo(B)$-module isomorphism $B\to pB$ (using torsion-freeness of $B$), these submodules form a lattice isomorphic to $\FIo(B)$.  This proves \eqref{eq:lattice-embed}.
\end{proof}

\begin{remark}
Explicitly, the lattice isomorphism sends a fully invariant subgroup $K\le B$ to $D\oplus pK$.  Thus a description of all the intervals in Theorem~\ref{thm:FI-embedding} also describes $\FIo(B)$ for arbitrary torsion-free $B$.  This is the precise reduction established here.  It does not, by itself, prove that any specified collection of classical invariants is insufficient for such a description.
\end{remark}

\begin{example}\label{ex:no-swap}
For $A=\Z(p^\infty)\oplus\Z$, the proof of Theorem~\ref{thm:FI-embedding} gives
\[
Z_E(A)=\Z(p^\infty)\oplus p\Z,
\qquad Z_E(A/t(A))=\Z.
\]
In contrast, $Z_{\Endo(\Z)}(\Z)=0$ by Corollary~\ref{cor:rank-one}.  The induced $E$-action in Theorem~\ref{thm:transfer} therefore cannot be replaced by the natural action of $\Endo(A/t(A))$.
\end{example}

\section{Concluding remarks}

The representation in Corollary~\ref{cor:fuchs} characterizes all singular submodules in the unrestricted setting of Fuchs' question.  Theorems~\ref{thm:transfer} and \ref{thm:delta-kernel} identify the torsion contribution and the obstruction to lifting quotient elements simultaneously at all relevant primes.  Theorem~\ref{thm:ExtHom} then expresses the remaining submodule choices by the splitting of a specified extension.

These statements give a module-theoretic answer and an exact decomposition of its data.  For a more explicit classification in further classes of groups, the remaining tasks are to determine the relative singular submodule $Z_E(B)$ and to compute the extension data.  Remark~\ref{rem:torsion-free-boundary} records the limitation when $A$ is torsion-free.  The interval realization in Theorem~\ref{thm:FI-embedding} shows how fully invariant subgroup lattices enter such a classification.

\end{document}